\documentclass[10pt]{amsart}

\usepackage{amsmath,amsfonts,amsthm,amssymb}
\usepackage[utf8]{inputenc}
\usepackage[pdfborder={0 0 0 [0 0 ]}]{hyperref}
\usepackage{graphicx}
\usepackage{color}
\usepackage[ps,all,arc,rotate,cmtip]{xy}
\usepackage{verbatim}
\usepackage{pgf,tikz}
\usepackage{extarrows}
\usetikzlibrary{arrows}
\usepackage[alphabetic]{amsrefs}
\usetikzlibrary{arrows}
\usepackage{caption}
\usepackage[capitalise]{cleveref}

\long\def\symbolfootnote[#1]#2{\begingroup 
\def\thefootnote{\fnsymbol{footnote}}\footnote[#1]{#2}\endgroup}
\numberwithin{equation}{section}

\newtheorem*{theorem*}{Theorem}
\newtheorem{theorem}{Theorem}[section]

\newtheorem{lem}[theorem]{Lemma}
\newtheorem{thm}[theorem]{Theorem}

\newtheorem{prop}[theorem]{Proposition}

\newtheorem{cor}[theorem]{Corollary}

\theoremstyle{definition}

\newtheorem{remark}[theorem]{Remark}

\newtheorem{dfn}[theorem]{Definition}

\newtheorem{ex}[theorem]{Example}

\newtheorem{quest}[theorem]{Question}
\newtheorem{conj}[theorem]{Conjecture}
\newtheorem{notation}[theorem]{Notation}

\newtheoremstyle{maintheorem}{3mm}{3mm}{\itshape}{}{\bfseries}{}{.5em}{#1 \!\thmnote{#3}.}
\theoremstyle{maintheorem}
\newtheorem*{mainthm}{Theorem}

\newtheorem*{mainprop}{Proposition}

\makeatletter
\def\thm@space@setup{%
	\thm@preskip=4mm \thm@postskip=4mm
}
\makeatother

\newcommand{\N}{\mathcal{N}}
\newcommand{\Q}{\mathbb{Q}}
\newcommand{\R}{\mathbb{R}}
\newcommand{\RR}{\mathcal{R}}
\newcommand{\Z}{\mathbb{Z}}

\newcommand{\C}{\mathbb{C}}

\newcommand{\PPP}{\mathfrak{P}}
\newcommand{\U}{\mathcal{U}}
\renewcommand{\P}{\mathcal{P}}

\newcommand{\CC}{\mathcal{C}}

\newcommand{\PP}{\mathbb{P}}
\newcommand{\Wh}{\operatorname{Wh}}
\newcommand{\Hom}{\operatorname{Hom}}
\newcommand{\Map}{\operatorname{Map}}

\newcommand{\hull}{\operatorname{hull}}

\newcommand{\ab}{{\operatorname{ab}}}

\newcommand{\pr}{{\operatorname{pr}}}
\newcommand{\ev}{{\operatorname{ev}}}
\newcommand{\odd}{{\operatorname{odd}}}

\renewcommand{\bar}{\overline}

\newcommand{\norm}{\mathfrak{N}}

\newcommand{\im}{\operatorname{im}}

\renewcommand{\subset}{\subseteq}

\def\id{\mathrm{id}}

\def\s-{\smallsetminus}

\def\D{\mathcal{D}}

\newcommand{\tolabel}[1]{\xlongrightarrow{#1}}
\renewcommand{\det}{\operatorname{det}}
\newcommand{\tor}{\rho^{(2)}}
\newcommand{\betti}{b^{(2)}}
\newcommand{\ct}{\chi^{(2)}}

\renewcommand{\phi}{\varphi}

\def\s-{\smallsetminus}

\newcommand{\zg}{{\Z G}}

\newcounter{flocomments}

\begin{document}
	
	\title{The $L^2$-torsion polytope of amenable groups}
	\author{Florian Funke}
	\address{Mathematisches Institut der Universit\"{a}t Bonn, Endenicher Allee 60, 53115 Bonn, Germany}
	\email{ffunke@math.uni-bonn.de}
	
	\begin{abstract}
		We introduce the notion of groups of polytope class and show that torsion-free amenable groups satisfying the Atiyah Conjecture possess this property. A direct consequence is the homotopy invariance of the $L^2$-torsion polytope among $G$-CW-complexes for these groups. As another application we prove that the $L^2$-torsion polytope of an  amenable group vanishes provided that it contains a non-abelian elementary amenable normal subgroup.
	\end{abstract}
	
	\maketitle

\section{Introduction}

In \cite{FriedlLueck2015, FriedlLueck2015b} Friedl-Lück construct a new geometric invariant $P(X;G)$ called \emph{$L^2$-torsion polytope} for a $G$-CW-complex $X$ (satisfying a number of assumptions, see \cref{sub:pol}), which shares many features with the $L^2$-torsion $\tor(X;\N(G))$. It takes values in an \emph{integral polytope group} $\P_T(H_1(G)_f)$, which is defined as the Grothendieck group of integral polytopes in $H_1(G)_f\otimes_\Z\R$ up to translation. Here $H_1(G)_f$ denotes the free part of the first integral homology $H_1(G)$ of $G$. One of the main results of Friedl-Lück's theory states that if $X = \widetilde{M}$ is the universal cover of a $3$-manifold $M$ (satisfying a number of conditions), then $P(\widetilde{M};\pi_1(M))$ is the dual of the unit ball of the Thurston norm, see \cite[Theorem 3.35]{FriedlLueck2015b}.

The $L^2$-torsion polytope has the potential to be a powerful geometric invariant on groups. Namely, if $G$ is an $L^2$-acyclic group of type $F$ which satisfies the Atiyah Conjecture and has vanishing Whitehead group, one can define the $L^2$-torsion polytope of $G$ as 
\[P(G) = P(EG;G).\]
A forerunner version of the $L^2$-torsion polytope of groups was defined and examined by Friedl-Tillmann \cite{FriedlTillmann2015} in the special case where $G$ is a torsion-free group given by a presentation with two generators, one relation, and first Betti number $b_1(G) = 2$. They show that in this case $P(G)$ completely determines the BNS-invariant of Bieri-Neumann-Strebel \cite{Bierietal1987}. A similar result was obtained by Kielak and the author \cite[Corollary 6.4]{FunkeKielak2016} for some free-by-cyclic groups.

This paper is motivated by the following conjecture of Friedl-Lück-Tillmann \cite[Conjecture 6.4]{FriedlLueckTillmann2016}  about the $L^2$-torsion polytope of amenable groups. We mention that in the original formulation of the conjecture \emph{not virtually $\Z$} can be replaced with \emph{not isomorphic to $\Z$} since any torsion-free virtually $\Z$ group is in fact isomorphic to $\Z$.

\begin{conj}[Vanishing of the $L^2$-torsion polytope of amenable groups]\label{conj:polytope amenable}
	Let $G\neq \Z$ be an amenable group satisfying the Atiyah Conjecture. Suppose that $G$ is of type $F$ and that $\Wh(G) = 0$. Then we have for the $L^2$-torsion polytope  \[ P(G) = 0.\]
\end{conj}

By means of the polytope homomorphism that is essential in the definition of the $L^2$-torsion polytope, we introduce the notion of groups \emph{of $P\geq 0$-class} and the even stronger property \emph{of polytope class}. These notions are polytope analogues of the notion of $\det\geq 1$-class about the Fuglede-Kadison determinant. Our first theorem shows that these definitions are meaningful.

\begin{mainthm}[{ \ref{thm:amenable polytope class}} \normalfont (Polytope class and amenability)]
	Let $G$ be a torsion-free amenable group satisfying the Atiyah Conjecture such that $H_1(G)_f$ is finitely generated. Then $G$ is of polytope class.
\end{mainthm}

It is worthwhile noting that for group of $P\geq 0$-class the $L^2$-torsion polytope is a $G$-homotopy invariant (rather than just a \emph{simple} $G$-homotopy invariant) of free finite $L^2$-acyclic $G$-CW-complexes and that therefore the condition that its Whitehead group vanishes is not necessary for $P(G)$ to be well-defined. We refer to  \cref{lem:polytope class} for more details on this remark.

We then adapt a strategy of Wegner for proving a vanishing result for the $L^2$-torsion of amenable groups \cite{Wegner2009} and obtain the following partial solution to \cref{conj:polytope amenable}.

\begin{mainthm}[{ \ref{thm:vanishing polytope}} \normalfont (Vanishing $L^2$-torsion polytope)]
	Let $G$ be a group of type $F$ which is of $P\geq 0$-class. Suppose that $G$ contains a non-abelian elementary amenable normal subgroup. Then $G$ is $L^2$-acyclic and we have
	\[ P(G) = 0.\]
	In particular, the $L^2$-torsion polytope of a non-cyclic elementary amenable group of type $F$ \mbox{vanishes.}
\end{mainthm}

Beyond elementary amenable groups, we provide at least some evidence for \cref{conj:polytope amenable}. In the following proposition, $*$ denotes the involution on the polytope group induced by reflection about the origin (see \cref{sub:pol groups}), and  $\norm$ denotes the seminorm homomorphism introduced in \cref{def:seminorm map}.

\begin{mainprop}[ \ref{prop: amenable}]
	Let $G\neq\Z$ be an amenable group of type $F$ satisfying the Atiyah Conjecture. Then $P(G)$ lies in the kernel of the seminorm homomorphism $\norm\colon \P_T(H_1(G)_f)\to \Map(H^1(G;\R),\R)$ and there is a polytope $P$ such that in $\P_T(H_1(G)_f)$ we have
	\[ P(G) = P-*P. \]
\end{mainprop}
\medskip

\subsection*{Acknowledgements} 

The author was supported by the Max Planck Institute for Mathematics in Bonn and the Deutsche Telekom Stiftung. We are grateful to Stefan Friedl, Fabian Henneke, Dawid Kielak, and Wolfgang L\"uck for many fruitful discussions and to the organizers of the \emph{New directions in $L^2$-invariants} workshop at the Hausdorff Institute for Mathematics in Bonn, where some of the ideas for this article were born. We also thank the referee for helfpul comments and various hints to formerly unmentioned references.

\newpage
\tableofcontents

\section{Background on the $L^2$-torsion polytope}\label{ch:preliminaries}

\subsection{The Atiyah Conjecture and $\D(G)$} The construction and our analysis of the $L^2$-torsion polytope requires some knowledge about the Atiyah Conjecture. If $R$ is a ring and $A\in M_{m,n}(R)$ is a matrix, then we let throughout $r_A\colon R^m\to R^n$ denote the $R$-homomorphism (of left $R$-modules) given by right multiplication with $A$.

\begin{conj}[Atiyah Conjecture]
	A torsion-free group $G$ satisfies the \emph{Atiyah Conjecture} (with rational coefficients) if for any matrix $A\in M_{m,n}(\Q G)$ we have 
	\[ \dim_{\N(G)}\big(\ker (r_A\colon  \N(G)^m\to \N(G)^n)\big) \in \Z.\]
\end{conj}

Here $\N(G)$ is the group von Neumann algebra of $G$ and $\dim_{\N(G)}$ denotes the dimension function on $\N(G)$-modules, see \cite[Definition 1.1 and Definition 6.20]{Lueck2002}. For a survey on the status of the Atiyah Conjecture we refer to \cite[Theorem 3.2]{FriedlLueck2015}. In order to explain its relevance in our context we need the following objects.

\begin{dfn}[$\U(G)$ and $\D(G)$]
	Let $\U(G)$ denote the algebra of operators affiliated to $\N(G)$, see \cite[Chapter 8]{Lueck2002}. Algebraically, this is the Ore localization of $\N(G)$ with respect to the set of weak isomorphisms, see \cite[Theorem 8.22 (1)]{Lueck2002}.
	
	Let $\D(G)$ be the smallest subring of $\U(G)$ which contains $\Q G$ and is \emph{division closed}, meaning that every element of $\D(G)$ which is a unit in $\U(G)$ is already a unit in $\D(G)$.
\end{dfn}

Thus we obtain a rectangle of inclusions
\[ \xymatrix{
	\Q G \ar[d]\ar[r] & \N(G)\ar[d]\\
	\D(G)\ar[r] & \U(G),
}\]
and using these rings we recall the following result.

\begin{prop}\label{D:skew field}
	A torsion-free group $G$ satisfies the Atiyah Conjecture if and only if $\D(G)$ is a skew-field.
\end{prop}
\begin{proof}
	See \cite[Lemma 10.39]{Lueck2002}.
\end{proof}

The next theorem, which combines results of Linnell and Tamari, is the central reason why the $L^2$-torsion polytope is tractable for amenable groups. 

\begin{thm}[$\D(G)$ of amenable groups]\label{D:amenable} 
	Any torsion-free elementary amenable group satisfies the Atiyah Conjecture.
	
	Moreover, if $G$ is a torsion-free amenable group satisfying the Atiyah Conjecture, then $\Q G$ satisfies the Ore condition with respect to $T = \Q G\s-\{0\}$ and there is an isomorphism $\D(G) \cong T^{-1}\Q G$. In particular, $\D(G)$ is flat over $\Q G$.
\end{thm}
\begin{proof}	
	The first part is due to Linnell \cite[Theorem 2.3]{Linnell2006}, see also \cite[Theorem 1.2]{KrophollerEtal2006}. 
	
	The fact that $\Q G$ satisfies the Ore condition with respect to $T$ goes back to Tamari \cite{Tamari1957}, see also \cite[Example 8.16 and Lemma 10.15]{Lueck2002} for a proof. Recalling the notion of division closure, it is then easy to see that the inclusion $\Q G\to \D(G)$ localizes to an isomorphism $T^{-1}\Q G\tolabel{\cong} \D(G)$.
\end{proof}

If $R$ is a ring and $0\to K\to G\tolabel{p} Q\to 0$ is a group extension, then any choice of (set-theoretic) section $s\colon Q\to G$ for $p$ induces an isomorphism 
\begin{equation}\label{eq:splitting}
	RG \cong (RK)*Q.
\end{equation} 
Here the right-hand side denotes a crossed product ring of $Q$ with coefficients in $RK$. We refer to \cite[Section 10.3.2]{Lueck2002} for a survey on crossed product rings and \cite[Example 10.53]{Lueck2002} for the details of the above statement. Here and henceforth we suppress the structure maps of crossed product rings from the notation. It will play an important role for us that $\D(G)$ shares similar structural properties. More precisely, we have

\begin{lem}[$\D(G)$ and extensions]\label{lem:DG and extensions}
	Let $G$ be a torsion-free group satisfying the Atiyah Conjecture. Let $0\to K\to G\tolabel{p} H\to 0$ be a group extension such that $H$ is finitely generated free-abelian. Then $K$ satisfies the Atiyah Conjecture and any choice of (set-theoretic) section $s\colon H\to G$ for $p$ determines a crossed product ring $\D(K)*H$ together with an inclusion $\D(K)*H\subset \D(G)$ which restricts to the isomorphism $ (\Q K)*H\cong \Q G$ of (\ref{eq:splitting}). Moreover, $\D(K)* H$ satisfies the Ore condition with respect to $T = (\D(K)* H)\s-\{0\}$, and the inclusion induces a $\D(K)$-isomorphism 
	\begin{equation}\label{eq:D iso}
		  T^{-1}(\D(K)* H) \cong \D(G).
	\end{equation}
	
	If $H$ is infinite cyclic, then $\D(K)* H$ is isomorphic to the ring $\D(K)_t[u^{\pm}]$ of twisted Laurent polynomials, where the twisting $t$ depends on $s$.
\end{lem}
\begin{proof}
	See \cite[Theorem 3.6 (3)]{FriedlLueck2015} and \cite[Example 10.54]{Lueck2002}, where also twisted Laurent polynomial rings are treated in detail.
\end{proof}

\subsection{Weak $K_1$-groups and universal $L^2$-torsion}
\label{sub:universal l2}

Let $G$ be a torsion-free group satisfying the Atiyah Conjecture. Define the \emph{weak $K_1$-group} $K_1^w(\Z G)$ as the abelian group whose generators $[f]$ are $\Z G$-maps $f\colon\Z G^n\to \Z G^n$ that become invertible over $\D(G)$, subject to the following relations:
If $f, g\colon \Z G^n\to\Z G^n$ are two such $\Z G$-maps, then require
\begin{equation}\label{reidemeister1} [g\circ f] = [f] + [g]. \end{equation}

If $f\colon \Z G^m\to\Z G^m,\: g\colon \Z G^n\to \Z G^n, \:h\colon \Z G^n\to\Z G^m$ are $\Z G$-maps such that $f$ and $g$ become invertible over $\D(G)$, then we require the relation
\begin{equation}\label{reidemeister2} \left[\begin{pmatrix}
f & h\\
0 & g
\end{pmatrix}\right] = [f] + [g].\end{equation}

This definition coincides with \cite[Definition 1.1]{FriedlLueck2015b} since $f\colon\Z G^n\to \Z G^n$ becomes invertible over $\D(G)$ if and only if $f$ induces a weak isomorphism $L^2(G)^n\to L^2(G)^n$. This follows from \cite[Lemma 1.21]{FriedlLueck2015b} and \cite[Lemma 10.39]{Lueck2002}.

We define the \emph{reduced weak $K_1$-group} and the \emph{weak Whitehead group} as the quotients
\begin{gather*}
	\widetilde{K}_1^w(\Z G) = K_1^w(\Z G)/\{[\pm\id\colon \Z G\to\Z G]\};\\
	\Wh^w(G)= K_1^w(\Z G)/\{[r_{\pm g}\colon \Z G\to \Z G]\:\mid g\in G\}.
\end{gather*}

There are obvious maps
\begin{gather*}
\widetilde{K}_1(\Z G) \to \widetilde{K}_1^w(\Z G)  \to\widetilde{K}_1(\D(G));\\
\Wh(G)\to \Wh^w(G) \to  K_1(\D(G))/\{[\pm g]\:\mid g\in G\}.
\end{gather*}

Recall that for any associative unital ring $R$ an $R$-chain complex $C_*$ is \emph{finite} if each chain module is finitely generated and only finitely many chain modules are non-trivial. It is \emph{based free} if each chain module is a free $R$-module and equipped with an equivalence class of $R$-basis, where two $R$ bases $B, B'$ are \emph{equivalent} if there exists a bijection $\sigma\colon B\to B'$ such that $\sigma(b) = \pm b$ for all $b\in B$. It is \emph{contractible} if there is a chain homotopy $\id_{C_*}\simeq 0$. If $C_*$ is a based free finite contractible $R$-chain complex, then we denote its Reidemeister torsion by $\tau(C_*)\in\widetilde{K}_1(R)$. Likewise we denote the Whitehead torsion of a $G$-homotopy equivalence $f\colon X\to Y$ of finite free $G$-CW-complexes by $\tau(f)\in\Wh(G)$.

A $\zg$-chain complex is \emph{$L^2$-acyclic} if all $L^2$-Betti numbers \[\betti_n(C_*;\N(G)) = \dim_{\N(G)} H_n( \N(G)\otimes_\zg C_*)\] 
vanish. For any based free finite $L^2$-acyclic $\Z G$-chain complex $C_*$ Friedl-Lück construct a \emph{universal $L^2$-torsion} 
\[\tor_u(C_*;\N(G)) \in \widetilde{K}_1^w(\Z G).\] 
Its construction is an adaption of the Reidemeister and Whitehead torsion to the $L^2$-setting. We briefly recall the definition of Reidemeister torsion here in order to give a flavour of these invariants. Let $K_1(\Z G)$ be the abelian group whose generators $[f]$ are $\Z G$-automorphisms $f\colon P\to P$ of finitely generated projective $\zg$-modules and whose relations are the same as for $K_1^w(\Z G)$, see (\ref{reidemeister1}) and (\ref{reidemeister2}). A $\zg$-chain complex $C_*$ is \emph{contractible} if $C_*$ admits a \emph{chain contraction}, i.e., a sequence of $\zg$-maps $\gamma_n\colon C_n\to C_{n+1}$ such that $c_{n+1}\circ \gamma_n + \gamma_{n-1}\circ c_n = \id_{C_n}$, where $c_n\colon C_n \to C_{n-1}$ denotes the differential. If $C_*$ is contractible, then its \emph{Reidemeister torsion} 
\[ \rho(C_*)\in \widetilde{K}_1(\Z G)\]
is defined as the class of the $\zg$-isomorphism
\[ c+\gamma\colon \bigoplus_{n\in \Z} C_{2n+1}\to \bigoplus_{n\in \Z} C_{2n}.\]
As further reference for algebraic $K$-theory and torsion invariants we recommend \cite{Silvester1981} or \cite{Rosenberg1994}, where it is proved that $c+\gamma$ is indeed a $\zg$-isomorphism and that its class in $\widetilde{K}_1(\Z G)$ does not depend on the choice of $\gamma$.

The passage from Reidemeister torsion to universal $L^2$-torsion is achieved by replacing \emph{chain contraction} with the weaker and more technical notion \emph{weak chain contraction}, see \cite[Definition 1.4]{FriedlLueck2015b}. Possessing a weak chain contraction turns out to be equivalent to being $L^2$-acyclic, see \cite[Lemma 1.5]{FriedlLueck2015b}. This is why the universal $L^2$-torsion is defined for $L^2$-acyclic chain complexes.

By \cite[Remark 1.16]{FriedlLueck2015b} the universal $L^2$-torsion deserves its name in the sense that it encapsulates all other $L^2$-torsion invariants, including the (classical) $L^2$-torsion $\tor(C_*;\N(G)) \in \R$, twisted $L^2$-torsion functions \cite{DuboisEtal2014, DuboisEtal2015, DuboisEtal2015b} and twisted $L^2$-Euler characteristics \cite{FriedlLueck2015}. 

If $X$ is a finite free $L^2$-acyclic $G$-CW-complex, then applying this to the cellular $\zg$-chain complex $C_*(X)$ produces the \emph{universal $L^2$-torsion of $X$}
\[\tor_u(X;\N(G)) \in \Wh^w(G).\] 
Its main properties are collected in \cite[Theorem 2.5]{FriedlLueck2015b}. We point out two of its properties that we need in this paper.

First, given a $G$-homotopy equivalence $f\colon X\to Y$ between finite free $L^2$-acyclic $G$-CW-complexes, then
\begin{equation}\label{eq:hom invariance}
\tor_u(Y;\N(G)) - \tor_u(X;\N(G)) = \zeta(\tau(f)).
\end{equation}
where $\zeta\colon \Wh(G)\to \Wh^w(G)$ is the obvious homomorphism.

We include the second statement here for future reference as a small lemma.

\begin{lem}\label{lem:torsion}
	Let $C_*$ be a finite based free $L^2$-acyclic $\zg$-chain complex. Then $\D(G)\otimes_{\zg} C_*$ is a contractible $\D(G)$-chain complex, and the canonical homomorphism $i\colon \widetilde{K}_1^w(\Z G)  \to\widetilde{K}_1(\D(G))$ satisfies
	\begin{equation}\label{eq:whitehead}
		i\big(\tor_u(C_*;\N(G))\big) = \tau(\D(G)\otimes_{\zg}C_*).
	\end{equation}
\end{lem} 
\begin{proof}
	The chain complex $\D(G)\otimes_{\zg}C_*$ is contractible by \cite[Lemma 1.21]{FriedlLueck2015b}.
	
	Let $R$ be any associative unital ring and $E_*$ a finite based free contractible $R$-chain complex. If $u_*\colon E_*\to E_*$ is a chain isomorphism and $\gamma_*\colon u_*\simeq 0_*$ is a chain homotopy such that $\gamma_n\circ u_n = u_{n+1}\circ\gamma_n$, then we have an equality
	\begin{equation} \label{eq:torsion formula}
		\tau(E_*) = [(uc+ \gamma)_\odd] - [u_\odd] \in \widetilde{K}_1(R).
	\end{equation}
	This follows in exactly the same way as the argument leading to \cite[Equation (1.8)]{FriedlLueck2015b}. Now the desired equation (\ref{eq:whitehead}) follows from this by comparing (\ref{eq:torsion formula}) with the definition of universal $L^2$-torsion \cite[Definition 1.7]{FriedlLueck2015b}.
\end{proof}

\subsection{Integral polytope groups}\label{sub:pol groups}

Let $H$ be a finitely generated free-abelian group. An integral polytope in $V_H = H\otimes_\Z\R$ is the convex hull of finitely many points in $H$, considered as a lattice in $V_H$. The \emph{Minkowski sum} of two integral polytopes $P$ and $Q$ in $V_H$ is defined by pointwise addition, i.e.,
\[ P+Q = \{p+q\in V_H\mid p\in P, q\in Q \}.\]
Denote by $\PPP(H)$ the commutative monoid of all integral polytopes in $V_H$ with the Minkowski sum as addition. It is cancellative, see e.g. \cite[Lemma 3.1.8]{Schneider1993}. Define the \emph{integral polytope group} $\P(H)$ to be the Grothendieck group associated to this commutative monoid. Thus elements are given by formal differences $P-Q$ of integral polytopes $P,Q\in \PPP(H)$, and two such differences $P-Q$, $P'-Q'$ are equal if and only if $P+Q' = P'+Q$ holds as subsets in $V_H$.

There is an injection of abelian groups
\begin{equation} \label{polytope injection}
H\to \P(H),\;\; h\mapsto \{h\}
\end{equation}
and we let $\P_T(H)$ be the cokernel of this map. The subscript $T$ stands for \emph{translation} since two polytopes become identified in $\P_T(H)$ if and only if there is a translation on $V_H$ mapping one bijectively to the other. We let $\PPP_T(H)$ be the image of the composition $\PPP(H)\to \P(H) \to \P_T(H)$.

The group $\P(H)$ carries a canonical involution induced by reflection about the origin, i.e.,
\begin{equation}\label{def:involution}
*\colon \P(H)\to \P(H),\;\; P\mapsto *P = \{-p\mid p\in P\}.
\end{equation}
This involution descends to an involution $*\colon\P_T(H)\to\P_T(H)$.

A homomorphism $f\colon H\to H'$ of finitely generated free-abelian groups induces homomorphisms 
\begin{gather*}
\P(f)\colon \P(H)\to \P(H');\\
\P_T(f)\colon \P_T(H)\to \P_T(H')
\end{gather*}
by sending the class of a polytope $P$ to the class of the polytope $f(P)$. If $f$ is injective, then both $\P(f)$ and $\P_T(f)$ are easily seen to be injective as well. Thus if $G\subseteq H$ is a subgroup, then we will always view $\P(G)$ (respectively $\P_T(G)$) as a subgroup of $\P(H)$ (respectively $\P_T(H)$).

\begin{ex}\label{ex:polytope in dim 1}
	Integral polytopes in $V_\Z = \R$ are just intervals $[m,n]\subseteq\R$ starting and ending at integral points. Thus we have $\P(\Z) \cong \Z^2$, where an explicit isomorphism is given by sending the class $[m,n]$ to $(m, n-m)$. Under this isomorphism, the involution corresponds to $*(k,l) = (-l-k, l)$. Similarly, $\P_T(\Z) \cong\Z$, where an explicit isomorphism is given by sending the element $[m,n]$ to $n-m$. The involution $*$ on $\P_T(\Z)$ is the identity.
\end{ex}

The structure of the integral polytope group was studied in detail by Cha-Friedl and the author \cite{ChaFriedlFunke2015} and by the author \cite{Funke2016}.

\subsection{The polytope homomorphism}

Let $G$ be a torsion-free group satisfying the Atiyah Conjecture such that $H_1(G)_f$, the free part of the first integral homology $H_1(G)$ of $G$ is finitely generated. Under these conditions, Friedl-Lück \cite[Section 6.2]{FriedlLueck2015b} define a \emph{polytope homomorphism}
\[\PP\colon K_1^w(\Z G)\to \P(H_1(G)_f). \]
Earlier versions of it had at least implicitly been considered for torsion-free elementary amenable groups \cite{FriedlHarvey2007, Friedl2007}. The polytope homomorphism is constructed as a composition 
\begin{equation}\label{eq:pol hom}
	K_1^w(\zg) \tolabel{i}  K_1(\D(G)) \tolabel{\det_{\D(G)}} \D(G)^\times_\ab \tolabel{P}\P(H_1(G)_f),
\end{equation}
where the first map is the canonical map, the second is the Dieudonné determinant \cite{Dieudonne1943} which is in fact an isomorphism (see \cite[Corollary 2.2.6]{Rosenberg1994} or \cite[Corollary 4.3]{Silvester1981}), and the third relies on the structural properties of $\D(G)$ given in \cref{lem:DG and extensions}. More precisely we let $K$ be the kernel of the projection $\pr\colon G\to H_1(G)_f = H$ and define
\[  P'\colon \D(K)* H\s-\{0\}\to \PPP(H)\] 
as follows: Given a non-trivial element $x = \sum_{h\in H} x_h\cdot h\in\D(K)* H$ we let $P'(x)$ be the convex hull of the set $\{h\in H \mid x_h\neq 0\}$. Then $P'$ is a homomorphism of monoids and induces a group homomorphism 
\[  P'\colon  \big(T^{-1}(\D(K)* H)\big)^\times_\ab\to \P(H),\;\; t^{-1}s\mapsto P'(s)-P'(t).\] 
Now we let $P$ be the composition
\begin{equation}\label{eq:P}
  P\colon \D(G)^\times_\ab \tolabel{\cong }  \big(T^{-1}(\D(K)* H)\big)^\times_\ab\tolabel{P'} \P(H),
\end{equation}
where the first map is the isomorphism appearing in \cref{lem:DG and extensions}. We will denote the induced maps
\begin{gather*}
	\PP\colon \widetilde{K}_1^w(\Z G)\to \P_T(H_1(G)_f)\\
	\PP\colon \Wh^w(G)\to \P_T(H_1(G)_f)
\end{gather*}
by the same symbol.

\begin{notation}\label{notation P(x)}
	For non-trivial $x\in \zg$ we denote the image of the class of $x$ in $\D(G)^\times_\ab$ under the map $P$ simply by $P(x)\in \P_T(H_1(G)_f)$. This is the same as $\PP\big( [r_x\colon \zg\to\zg]\big)$.
\end{notation}

\subsection{The $L^2$-torsion polytope}\label{sub:pol} The definition of our main object of study is now fairly simple.

\begin{dfn}[$L^2$-torsion polytope]\label{def:l2 polytope}
	Let $G$ be a torsion-free group satisfying the Atiyah Conjecture such that $H_1(G)_f$ is finitely generated. Let $X$ be a finite free $L^2$-acyclic $G$-CW-complex. Then the \emph{$L^2$-torsion polytope of $X$} is defined as the image of the negative of its universal $L^2$-torsion under the polytope homomorphism, i.e.,
	\[P(X;G) = \PP\big(-\tor_u(X;\N(G))\big)\in \P_T(H_1(G)_f).\]
	
	Let $G$ be a group of type $F$ satisfying the Atiyah Conjecture. If $G$ is $L^2$-acyclic and satisfies $\Wh(G) = 0$, then we may define the \emph{$L^2$-torsion polytope of $G$} to be
	\[ P(G) = P(EG;G) \in \P_T(H_1(G)_f).\]
\end{dfn}

 \begin{remark}[Assumptions appearing in \cref{def:l2 polytope}]
 	The assumption $\Wh(G) = 0$ appearing above ensures that the $L^2$-torsion polytope of groups is well-defined, see (\ref{eq:hom invariance}). Conjecturally, however, this assumption is obsolete: Any group of type $F$ is torsion-free, and it is conjectured that the Whitehead group of any torsion-free group vanishes, see \cite[Conjecture 3]{LueckReich2005}. There is also no counterexample to the Atiyah Conjecture known. Thus the $L^2$-torsion polytope is potentially an invariant for all $L^2$-acyclic groups of type $F$.
 	
 	Within the class of amenable groups all torsion-free virtually solvable groups are known to have trivial Whitehead group since they satisfy the $K$-theoretic Farrell-Jones Conjecture, as proved by Wegner \cite{Wegner2013}.
 \end{remark}

\section{Groups of $P\geq 0$-class}\label{ch:pol class}

In this section we introduce a polytope analogue of the notion $\det\geq 1$-class concerning the Fuglede-Kadison determinant \cite[Definition 3.112]{Lueck2002}. First we need a partial order on polytope groups.

\begin{dfn}[Partial order on polytope groups]\label{def:partial order}
	Let $H$ be a finitely generated free-abelian group. We define a partial order on $\P(H)$ by declaring
	\[ P-Q\leq P'-Q' \;\text{ if and only if }\; P+Q'\subset P'+Q.\]
	Likewise, we define a partial order on the translation quotient $\P_T(H)$ by declaring
	\[ P-Q\leq P'-Q' \;\text{ if and only if }\; P+Q'\subset P'+Q\text{ up to translation}.\]
\end{dfn}

It is easy to see that this definition does not depend on the choice of representatives.

\begin{dfn}[$P\geq 0$-class and polytope class]
	A group $G$ is \emph{of $P\geq 0$-class} if it is torsion-free, satisfies the Atiyah Conjecture, $b_1(G)<\infty$, and we have for any matrix $A\in M_{n,n}(\Z G)$ which becomes invertible over $\D(G)$ that
	\[ \PP\big([r_A\colon \Z G^n\to \Z G^n]\big) \geq 0\]
	in $\P_T(H_1(G)_f)$. We call $G$ \emph{of polytope class} if $\PP\big([r_A\colon \Z G^n\to \Z G^n]\big)$ is even represented by a polytope, i.e., it lies in the submonoid $\PPP_T(H_1(G)_f)\subseteq\P_T(H_1(G)_f)$ of integral polytopes up to translation.
\end{dfn}

\begin{ex}
	\begin{enumerate}
		\item A finitely generated free-abelian group $H$ is of polytope class since the Dieudonné determinant $\det_{\D(H)}(A)$ coincides with the determinant $\det_{\Z H}(A)$ over the commutative ring $\Z H$ and is therefore represented by an element in $\Z H$. Hence $\PP\big([r_A\colon \Z H^n\to \Z H^n]\big)$ is  represented by a polytope.
		\item If $G$ is a torsion-free group satisfying the Atiyah Conjecture such that $H_1(G)_f$ is of rank at most $1$, then $G$ is of polytope class. Namely, let $\D(K)_t[u^\pm]\subset\D(G)$ be a subring determined by a generator of $\Hom(G,\Z)$, as explained in \cref{lem:DG and extensions}. Then it follows by virtue of the Euclidean function on $\D(K)_t[u^\pm] $ given by the degree that $\det_{\D(G)}(A)$ is represented by an element in $\D(K)_t[u^\pm]$. (A similar argument will be used in the proof of \cref{thm:amenable polytope class} where more details can be found.) Thus $\PP\big([r_A\colon \Z G^n\to \Z G^n]\big)$ is represented by an interval.
	\end{enumerate}
\end{ex}

We know from (\ref{eq:hom invariance}) that the $L^2$-torsion polytope is a simple homotopy invariant of free finite $L^2$-acyclic $G$-CW-complexes. This can be strengthened if $G$ is of $P\geq 0$-class.

\begin{lem}\label{lem:polytope class}
	Let $G$ be a group of $P\geq 0$-class. Then the composition
	\[ \Wh(G) \tolabel{\zeta} \Wh^w(G) \tolabel{\PP} \P_T(H_1(G)_f)\]
	is trivial. Moreover, the $L^2$-torsion polytope is a homotopy invariant of free finite $L^2$-acyclic $G$-CW-complexes.
\end{lem}
\begin{proof}
	An element in the image of $\zeta$ is of the form $[r_A\colon \zg^n\to\zg^n]$ for a matrix $A\in M_{n,n}(\zg)$ which has an inverse $A^{-1}\in M_{n,n}(\zg)$. Since $G$ is of $P\geq 0$-class, we have
	\[ 0 = \PP([\id]) = \PP\big([r_A]\big) + \PP\big([r_{A^{-1}}]\big) \geq 0,\]
	and hence $\PP([r_A]) = 0$. The 'moreover' part immediately follows from this because of (\ref{eq:hom invariance}).
\end{proof}

\begin{remark}[Extension of $P(G)$ to groups of $P\geq 0$-class]
	\cref{lem:polytope class} allows us to drop $\Wh(G) = 0$ from the list of conditions in the definition of the $L^2$-torsion polytope $P(G)$ of groups (see \cref{def:l2 polytope}), provided that $G$ is of $P\geq 0$-class. Put differently, we can extend the definition of $P(G)$ to groups $G$ which are of type $F$ and of $P\geq 0$-class. We will take this into account in the formulations for the rest of this paper.
\end{remark}

\section{Polytope class and amenability}\label{ch:pol class and amenabl}

The goal of this section is to prove the following result.

\begin{thm}[Polytope class and amenability]\label{thm:amenable polytope class}
	Let $G$ be a torsion-free amenable group satisfying the Atiyah Conjecture such that $H_1(G)_f$ is finitely generated. Then $G$ is of polytope class.
\end{thm}

Its proof requires some preparation. Our main technical tool going into the proof are face maps.

\begin{dfn}[Faces and face maps]
	Let $H$ be a finitely generated free-abelian group and $P\subset V_H = H\otimes_\Z\R$ an integral polytope. Take $\phi\in\Hom(H,\Z)$ which we also view as an element in $\Hom(H,\R)=\Hom_\R(V_H,\R)$. Then we call
	\[ F_\phi(P) = \{p\in P\mid \phi(p) = \max\{ \phi(q)\mid q\in P\} \}\]
	the \emph{face of $P$ in $\phi$-direction}, see also \cref{fig:partition}. A subset $F\subseteq P$ is called a \emph{face} if $F_\phi(P) = F$ for some $\phi\in\Hom(H,\Z)$.
	
	A face of an integral polytope is an integral polytope in its own right, and it is straightforward to check that $F_\phi(P+Q) = F_\phi(P)+F_\phi(Q)$. These two observations imply that we obtain a homomorphism
	\begin{equation*} 
	F_\phi\colon \P(H)\to \P(H), \;\; P\mapsto F_\phi(P)
	\end{equation*}
	that we call \emph{face map in $\phi$-direction}. There is an induced face map (denoted by the same symbol)
	\begin{equation*} 
	F_\phi\colon \P_T(H)\to \P_T(H)
	\end{equation*}
	whose image is contained in the subgroup $\P_T(\ker \phi)$.
\end{dfn} 

	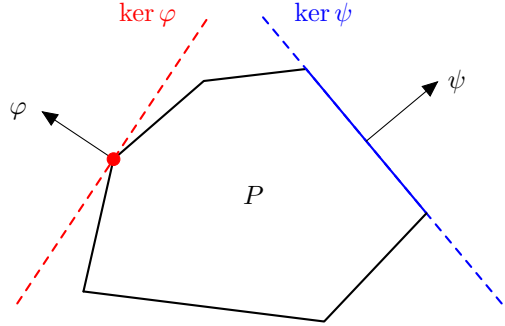
\begin{figure}[h]
	\captionsetup{width=0.7\textwidth}
	\centering
	\begin{tikzpicture}[line cap=round,line join=round,>=triangle 45,x=1.0cm,y=1.0cm, scale = 0.8]
	\clip(-2.5, -1.6) rectangle (6.5,4.2);
	\draw (1.5, 0.6) node[right = 0pt] {$P$} (1,0);
	\draw[thick]  (-1,-1) -- (-0.5, 1.2) -- (1, 2.5) -- (2.7, 2.7);
	\draw[thick, blue]  (2.7, 2.7) -- (4.7, 0.3);
	\draw[thick, blue, dashed]  (5.95, -1.2) -- (2.7, 2.7) -- (1.95, 3.6) node[right = 9pt] {$\ker\psi$} (1,0);
	\draw[thick]  (4.7, 0.3) -- (3, -1.5) -- (-1,-1);
	\draw[->]  (3.7, 1.5) -- (4.9, 2.5);
	\draw (4.9, 2.5) node[right = 0pt] {$\psi$} (1,0);
	
	\draw[->]  (-0.5, 1.2) -- (-1.7, 2.);
	\filldraw[red] (-0.5, 1.2) circle (3pt);
	\draw (-1.7, 2.) node[left= 2pt] {$\phi$} (1,0);
	\draw[thick, red, dashed]  (-2.1, -1.2) --  (1.1, 3.6) node[left = 9pt] {$\ker\phi$} (1,0);
	\end{tikzpicture}
	\caption{A polytope $P$ and two morphisms $\phi$ and $\psi$ indicated by (translates of) their kernels and the directions in which they maximize. Here $F_\phi(P)$ is represented by the red vertex and $F_\psi(P)$ is represented by the blue edge.}\label{fig:partition}
\end{figure}

The first lemma is possibly well-known in polytope theory, but we were not able to find the statement nor an implicit proof in the literature. In any case, it might be helpful in other situations.

\begin{lem}[Detecting polytopes by their faces]\label{lem:polytope}
	Let $H$ be a finitely generated free-abelian group of rank at least $2$. Then $x\in \P(H)$ is represented by a polytope if and only if for every $\phi\in\Hom(H,\Z)$ the class $F_\phi(x)\in\P(H)$ is represented by a polytope.
\end{lem}
\begin{proof}		
	It suffices to prove this for $H = \Z^n$. Equip $V_H = \R^n$ with the standard inner product. The forward direction of the lemma is obvious. 
	
	For the backwards direction write $x = P-Q$ for integral polytopes $P$ and $Q$. By assumption $F_\phi(x) = F_\phi(P) - F_\phi(Q)$ is an integral polytope for any $\phi\in \Hom(H,\Z)$, say $S^\phi$, so $F_\phi(P) = F_\phi(Q)+S^\phi$. We can write 
	\[P = \{x\in V_H\mid \psi_i(x)\leq c_i\}\] 
	for certain $\psi_i\in\Hom(H,\Z)\subset \Hom_\R(V_H,\R)$ and $c_i\in\Z$ ($i = 1,...,k$). Then 
	\[S = \hull\left(\bigcup_{i=1}^k S^{\psi_i}\right)\] 
	is an integral polytope satisfying $P\subset Q+S$. The remainder of the proof will be occupied with proving $Q+S \subset P$ which will imply $x = P-Q = S$. This requires a number of steps. In the following, Greek letters will always denote elements in $\Hom(H,\Z)$ without explicitly saying this. Moreover, given a compact subset $A\subset V_H$ and $\phi$, we will use the shorthand notations
	\begin{gather*}
	A_\phi = F_\phi(A); \\
	\phi(A) = \max\{\phi(a)\mid a\in A\}.
	\end{gather*}
	First note that	we have for any $\phi$ and $\psi$
	\[ F_\phi(P_\psi) = P_\phi \cap P_\psi = F_\psi(P_\phi)\]
	provided that the intersection in the middle is non-trivial, and likewise for $Q$.
	
	\medskip
	\emph{Step 1:} If $\phi, \psi$ are such that $P_\phi\cap P_\psi$ is non-empty, then $Q_\phi\cap Q_\psi$ and $S^\phi\cap S^\psi$ are non-empty, and we have
	\[ P_\phi\cap P_\psi = \big(Q_\phi\cap Q_\psi \big) + \big(S^\phi\cap S^\psi \big).\]
	
	We first argue that $Q_\phi\cap Q_\psi$ is non-empty. Pick a vertex $p\in P_\phi\cap P_\psi$, and let $\alpha$ be such that $P_\alpha = p$. Then $p = P_\alpha = Q_\alpha + S^\alpha$, hence $Q_\alpha = q$ and $S^\alpha = s$ are just points. After translating $Q$, we may assume that $s= 0$ and $p=q$. Then for every $\beta$ such that $P_\beta$ contains $p$ we have $Q_\beta\subset P_\beta$ and $p\in Q_\beta$. This applies in particular to $\phi$ and $\psi$, hence $p\in Q_\phi\cap Q_\psi$.
	
	Now we compute
	\[ F_\phi(S^\psi) = F_\phi(P_\psi) - F_\phi(Q_\psi)  = F_\psi(P_\phi) - F_\psi(Q_\phi)  = F_\psi(S^\phi),\]
	hence $F_\phi(S^\psi) \subset S^\phi\cap S^\psi$ and $S^\phi\cap S^\psi$ is non-empty. We also have
	\begin{align*} 
	\big(S^\phi\cap S^\psi\big) + F_\phi(Q_\psi) &= 	\big(S^\phi\cap S^\psi\big) + \big(Q_\phi\cap Q_\psi\big)\\
	&\subset	\big(P_\phi\cap P_\psi\big) \\
	&= 	F_\phi(P_\psi).
	\end{align*}
	From this it follows that $S^\phi\cap S^\psi \subset F_\phi(S^\psi)$. Thus we proved $ F_\phi(S^\psi) = S^\phi\cap S^\psi$. Now we conclude 
	\begin{align*}
	P_\phi\cap P_\psi &= F_\phi(P_\psi)\\
	&= F_\phi(Q_\psi) + F_\phi(S^\psi)\\
	&= \big(Q_\phi\cap Q_\psi\big) + \big(S^\phi\cap S^\psi\big). 
	\end{align*}
	
	\medskip
	\emph{Step 2:} Let $v_0,v_1,v_2\in S^{n-1} \subset \R^n$ be such that $v_1$ lies on a geodesic path of length at most $\pi$ from $v_0$ to $v_2$ in $S^{n-1}$. Write $\phi_i = \langle v_i, \cdot\rangle\colon \R^n\to \R$. If $P$ is any polytope such that $P_{\phi_1} \cap P_{\phi_2}$ is non-trivial, then we have
	\[ \phi_0(P_{\phi_2}) = \phi_0(P_{\phi_1} \cap P_{\phi_2}).\]
	
	Pick an element $x\in P_{\phi_1} \cap P_{\phi_2}$ attaining the maximum on the right. Assume that we have 
	\[ \phi_0(P_{\phi_2}) > \phi_0(P_{\phi_1} \cap P_{\phi_2}).\]
	Then there exists $y\in P_{\phi_2}$ such that $	\phi_0(y) > \phi_0(x)$, $\phi_1(y) < \phi_1(x)$, and $\phi_2(y) = \phi_2(x)$. In other words,
	\begin{gather*}
		\langle y-x, v_0\rangle > 0;\\
		\langle y-x, v_1\rangle < 0;\\
		\langle y-x, v_2\rangle = 0
	\end{gather*}
	which cannot happen if $v_1$ lies on a geodesic path of length at most $\pi$ from $v_0$ to $v_2$.
	
	\medskip
	\emph{Step 3:} We have $S^\phi = S_\phi$.
	
	Let $\phi, \psi$ be arbitrary and write (up to scalar) $\phi = \langle v, \cdot \rangle$ and $\psi = \langle w, \cdot \rangle$ for unit vectors $v,w$. There is a sequence of unit vectors $v = v_0, v_1, ..., v_m = w$ running along a geodesic path of length at most $\pi$ from $v$ to $w$ in $S^{n-1}$ such that $P_{\phi_i}\cap P_{\phi_{i+1}}$ is non-trivial. For brevity write from now on $P_i=P_{\phi_i}$, $Q_i = Q_{\phi_i}$, and $S^i = S^{\phi_i}$. Then we have by Step 1 
	\begin{gather*}
	P_i\cap P_{i+1} = \big(Q_i\cap Q_{i+1} \big)+\big(S^i\cap S^{i+1}\big)
	\end{gather*}
	and by Step 2
	\begin{gather*}
	\phi(P_{i+1}) = \phi(P_i\cap P_{i+1});\\
	\phi(Q_{i+1}) = \phi(Q_i\cap Q_{i+1}).
	\end{gather*}
	This implies 
	\begin{align*}
		\phi(S^{i+1})  &= \phi(P_{i+1}) - \phi(Q_{i+1})\\
				&= \phi(P_i\cap P_{i+1}) -\phi(Q_i\cap Q_{i+1})\\
				&= \phi(S^i\cap S^{i+1}) \\
				&\leq \phi(S^i).
	\end{align*}
	Since this is true for all $i = 0,..., m-1$, we conclude $\phi(S^\psi) \leq \phi(S^\phi)$ and hence $S^\phi = S_\phi$.
	
	\medskip
	\emph{Step 4:} We have $Q+ S\subset P = \{x\in V_H\mid \psi_i(x)\leq c_i\}$. 
	
	Pick arbitrary $q\in Q$ and $s\in S$. With the aid of Step 3 we can calculate
	\begin{align*} 
	\psi_i(q+s) &= \psi_i(q)+\psi_i(s) \\
	& \leq \psi_i(Q_{\psi_i}) + \psi_i(S_{\psi_i}) \\
	& = \psi_i(Q_{\psi_i}) + \psi_i(S^{\psi_i}) \\
	& = \psi_i(P_{\psi_i}) = c_i
	\end{align*}
	for all $i$, and hence $q+s \in P$.
\end{proof}

We also need the following auxiliary gadget.

\begin{dfn}
	Let $H$ be a finitely generated free-abelian group and $G\subset H$ a subgroup. We consider $\PPP_T(G)$ as a submonoid of $\PPP_T(H)$. Then we let $\P_T(H,G)$ be the submonoid of $\P_T(H)$ containing all elements that can be written as a difference $P-Q$ for some $P\in \PPP_T(H)$ and $Q\in \PPP_T(G)$.
\end{dfn}

\begin{ex}
	\begin{enumerate}
		\item For any subgroup $G\subset H$ one has 
		\[ \PPP_T(H) = \P_T(H,0)\subset \P_T(H,G)\subset \P_T(H,H) = \P_T(H).\]
		We can interpret $\P_T(H,G)$ as interpolating between the monoid of integral polytopes and the integral polytope group.
		\item Let $H$ be of rank $2$ and let $G_1, G_2$ be two subgroups of rank $1$. If $G_i\cap G_j = 0$, then $\P_T(H,G_1)\cap \P_T(H,G_2) = \PPP_T(H)$.
	\end{enumerate}
\end{ex}

Motivated by the last example we propose the following problem.

\begin{quest}
	Let $H$ be a finitely generated free-abelian group and $G_1, G_2$ be two subgroups. Do we always have 
	\[\P_T(H,G_1)\cap \P_T(H,G_2) = \P_T(H, G_1\cap G_2)?\]
\end{quest}

If this question has an affirmative answer, then the next lemma, for which we provide a different argument, would immediately follow.

\begin{lem}\label{lem:intersection}
	Let $H$ be a finitely generated free-abelian group. Then 
	\[\bigcap_{\phi\in\Hom(H, \Z)} \P_T(H, \ker \phi) = \PPP_T(H).\]
\end{lem}
\begin{proof}
	We prove the statement by induction on the rank of $H$. The rank $1$ case is obvious.
	
	For the higher rank case, pick an element $x$ in the above intersection. For any homomorphism $\phi\colon H\to\Z$ we can find $P_\phi\in \PPP_T(H)$ and $Q_\phi\in \PPP_T(\ker\phi)$ such that $x = P_\phi-Q_\phi$. Fix some homomorphism $\alpha\colon H\to\Z$. Then
	\[ F_\alpha(x) = F_\alpha( P_\phi) - F_\alpha(Q_\phi) \in \P_T(\ker\alpha, \ker\alpha\cap \ker\phi).\]
	Since $\phi$ was arbitrary, we conclude
	\[ F_\alpha(x) \in \bigcap_{\phi\in\Hom(H,\Z)} \P_T(\ker\alpha, \ker\alpha\cap \ker\phi) = \bigcap_{\psi\in\Hom(\ker\alpha,\Z)} \P_T(\ker\alpha, \ker\psi)  .\]
	From the induction hypothesis we conclude $F_\alpha(x) \in \PPP_T(\ker\alpha)$. As this holds for every homomorphism $\alpha\colon H\to \Z$, we may apply the previous \cref{lem:polytope} to deduce that $x\in \PPP_T(H)$.
\end{proof}

Now we can tackle the main result of this section.

\begin{proof}[Proof of \cref{thm:amenable polytope class}]
	Recall from \cref{D:amenable} that $\Z G$ satisfies the Ore condition with respect to $T= \Z G\s-\{0\}$ and the inclusion induces an isomorphism $T^{-1} \Z G\tolabel{\cong} \D(G)$.
	
	Let $A\in M_{n,n}(\Z G)$ be a matrix which becomes invertible over $\D(G)$. If $H_1(G)_f = 0$, then there is nothing to prove. Otherwise let us fix some epimorphism $\phi\colon G\to \Z$ and denote its kernel by $K$. Consider the associated twisted Laurent polynomial ring $\D(K)_t[u^\pm]\subset \D(G)$ as in \cref{lem:DG and extensions}. The Euclidean function on $\D(K)_t[u^\pm]$ given by the degree allows us to transform $A$ to a triangular matrix $T$ over $\D(K)_t[u^\pm]$ by using the operations
	\begin{itemize}
		\setlength\itemsep{0mm}
		\item Permute rows or columns;
		\item Multiply a row on the right or a column on the left with an element of the form $y\cdot u^m$ for some non-trivial $y\in \D(K)$ and $m\in\Z$;
		\item Add a right $\D(K)_t[u^\pm]$-multiple of one row (resp. column) to another row (resp. column).
	\end{itemize}
	These operations change the class $[A] \in K_1(\D(G))$ by adding an element of the form $[y\cdot u^m]$ for some non-trivial $y\in \D(K)$ and $m\in\Z$. Since $\D(K) = (\Z K\s-\{0\})^{-1} \Z K$, we may then multiply $T$ with suitable elements in $\Z K$ to obtain a matrix over $\Z K_t[u^\pm] = \Z G$. This implies that there are elements $a\in\zg$ and $b\in \Z K\s-\{0\}$ such that we have in $K_1(\D(G))$
	\[[A] = [T] - [y\cdot u^m] =  [a\cdot b^{-1}] -[y\cdot u^m] .\]
	Since $P(u^m) = 0$ in $\P_T(H_1(G)_f)$, we have
	\[ \PP\big([r_A\colon \Z G^n\to \Z G^n]\big) = P(a) - P(b) - P(y)\in \P_T(H_1(G)_f, \ker\bar{ \phi}) \]
	for the epimorphism $\bar{\phi}\colon H_1(G)_f\to\Z$ induced by $\phi$. Since $\phi$ was arbitrary, we have
	\[ \PP\big([r_A\colon \Z G^n\to \Z G^n]\big) \in\bigcap_{\substack{\phi\in\Hom(G,\Z) \\\text{ surjective}}} \P_T(H_1(G)_f, \ker\bar{ \phi}).  \]
	By \cref{lem:intersection}, this intersection is equal to $\PPP_T(H_1(G)_f)$, and the proof is complete.
\end{proof}

\section{Polytope class and the $L^2$-torsion polytope}\label{ch:pol class and polytopes}

In this section we adapt Wegner's strategy built in \cite{Wegner2000, Wegner2009} to the setting of the $L^2$-torsion polytope. Together with the knowledge that torsion-free amenable groups are of polytope class, one of its applications will be the vanishing of the $L^2$-torsion polytope of every elementary amenable group of type $F$. In order to motivate our first lemma we give a rough idea of the argument:

Instead of localizing the group ring $\zg$ at $\zg\s-\{0\}$ in order to obtain $\D(G)$, we localize at a much smaller set $S\subset\zg$ in order to obtain an intermediate ring $\zg\subset S^{-1}\zg\subset \D(G)$. This set is small enough so that the polytope of invertible matrices over $S^{-1}\zg$ still satisfies $P\geq 0$, but it is large enough so that the localized cellular chain complex $S^{-1}C_*(EG)$ is already contractible. Combining these two facts makes the image of the Whitehead torsion of $S^{-1}C_*(EG)$ under an adjusted polytope homomorphism $K_1(S^{-1}\zg)\to \P_T(H_1(G)_f)$ computable. But this image coincides with the negative of the $L^2$-torsion polytope $P(G)$.

It is worthwhile mentioning that this kind of partial Ore localization technique was used for the first time by Rosset \cite{Rosset1984} in proving that the Euler characteristic of a group of type $F$ vanishes provided that it contains a non-trivial normal abelian subgroup.

\begin{lem}\label{lem:ore}
	Let $G$ be a group of type $F$ which satisfies the Atiyah Conjecture and $b_1(G)<\infty$. Suppose that $G$ contains a non-trivial abelian normal subgroup $A\subset G$ such that $A\cap\ker(\pr\colon G\to H_1(G)_f)\neq 0$. Then
	\[ S = \{ x\in \Z A\s-\{0\}\mid P(x)= 0\;\text{ in }\; \P_T(H_1(G)_f) \}.\]
	is a multiplicatively closed subset with respect to which $\Z G$ satisfies the Ore condition and such that $S^{-1}\Z = 0$ for the trivial $\Z G$-module $\Z$.
\end{lem}
\begin{proof}
	Since for any two elements $x,y\in \Z G$ we have $P(x\cdot y) = P(x)+ P(y)$, it is clear that $S$ is multiplicatively closed. The proof for the left and right Ore condition follows as in \cite[Proof of Theorem 5.4.5, Step 2 and 3]{Wegner2000}, see also \cite[Lemma 3.119]{Lueck2002}. We include the argument here for the sake of completeness. Note that the canonical involution on $\zg$ respects $S$, so it suffices to prove the right Ore condition. 
	
	Let $r\in\Z G, s\in S$ and fix a set of representatives $\{g_i\mid i\in I\}$ for the cosets $Ag\in A\backslash G$. Write $r = \sum_{i\in I} a_ig_i$ for certain $a_i\in \Z A$, where almost all $a_i$ vanish. Put $I' = \{i\in I\mid a_i\neq 0\}$. The element $s_i = g_isg_i^{-1}$ lies in $\Z A$ since $A$ is normal and
	$ P(s_i) = P(s) = 0$. These two facts imply $s_i\in S$.
	
	Define $s' = \prod_{i\in I'} s_i \in S$, $x_i = s'/s_i\in S$, and $r' = \sum_{i\in I'} x_ia_ig_i\in\zg$. Then we compute
	\begin{align*}
	s'\cdot r &= \sum_{i\in I'}  s'a_ig_i 
	= \sum_{i\in I'}  x_is_ia_ig_i
	= \sum_{i\in I'} x_ia_is_ig_i \\
	&= \sum_{i\in I'} x_ia_ig_isg_i^{-1}g_i
	= \sum_{i\in I'} x_ia_ig_is
	= r'\cdot s
	\end{align*}
	
	Finally we prove $S^{-1}\Z = 0$. Pick some non-trivial 
	\[a\in A\cap\ker(\pr\colon G\to H_1(G)_f)\neq 0\] 
	(this is the only part where we need this assumption). Then $P(1-a) = 0$ in $\P_T(H_1(G)_f)$, so $1-a$ lies in $S$. Since $1-a$ acts by multiplication with $0$ on $\Z$, we conclude $S^{-1}\Z = 0$.
\end{proof}

\begin{lem}\label{lem:vanishing}
	Let $G$ be a group of $P\geq 0$-class. Let $S\subseteq \Z G$ be a multiplicatively closed subset with respect to which $\Z G$ satisfies the Ore condition. Suppose that $P(s)= 0\text{ in } \P_T(H_1(G)_f)$ for all $s\in S$.
	
	If $X$ is a free finite $L^2$-acyclic $G$-CW-complex such that $S^{-1}H_n(X) = 0$, then 
	\[ P(X;G) = 0.\]
\end{lem}
\begin{proof}
	This is based on ideas appearing in \cite[Proof of Theorem 5.4.5, Step 4 and 5]{Wegner2000}, see also \cite[Lemma 3.114]{Lueck2002}.
	
	First we consider the following commutative diagram
	\[\xymatrix@R=5mm{
		\widetilde{K}_1^w(\zg) \ar[dr]^(.48){i} \ar@/^1.4pc/[drrr]^{\PP} & & \\
		& \widetilde{K}_1(\D(G)) \ar[r]^(.45){\det_{\D(G)}} & \D(G)^\times_\ab/\{\pm 1\} \ar[r]^(.51){P} & \P_T(H_1(G)_f)\\
		\widetilde{K}_1(S^{-1}\zg)\ar[ur]^(.53){j}  \ar@/_1.4pc/[urrr]^{\PP'} 
	}\]
	Here $i$ and $j$ denote the obvious maps, $\det_{\D(G)}$ is the Dieudonné determinant, $P$ is induced by the map defined in (\ref{eq:P}), $\PP$ denotes the composition of the top row (which is the polytope homomorphism), and $\PP'$ denotes the composition of the bottom row. 
	
	Let $A$ be an invertible $S^{-1}\zg$-matrix. By multiplying $A$ with a suitable $s\in S$ we obtain a $\zg$-matrix $B$ which is invertible over $S^{-1}\zg$ and thus also over $\D(G)$. Then we have $[A] = [B] - [s]$ in $\widetilde{K}_1(S^{-1}\zg)$ and $\PP'([B]) =  \PP([B])$. We assume that $P(s) = 0$ and that $G$ is of $P\geq 0$-class, so we have
	\begin{equation}\label{eq:positive}
	\PP'([A]) = \PP'([B]) - \PP'([s]) = \PP'([B]) - P(s)=  \PP([B])\geq 0.
	\end{equation}
	Since the same reasoning applies to $A^{-1}$, we have $\PP'([A]) = 0$ and thus $\PP' = 0$.
	
	Denote by $C_* = C_*(X)$ the cellular $\zg$-chain complex of $X$ equipped with some choice of cellular basis. By \cref{lem:torsion} the $\D(G)$-chain complex $\D(G)\otimes_{\zg} C_*$ is contractible and we have 
	\begin{equation*}
		i(\tor_u(C_*;\N(G))) = \tau(\D(G)\otimes_{\zg} C_*).
	\end{equation*} 
	Since localization is flat and $S^{-1}H_n(X) = 0$, the $S^{-1}\zg$-chain complex $S^{-1}C_* = S^{-1}\zg\otimes_\zg C_*$ is also contractible, and we have
	\begin{align*}
	j(\tau(S^{-1}C_*)) &=\tau(\D(G)\otimes_{S^{-1}\zg} S^{-1}C_*) \\
	&= \tau(\D(G)\otimes_{S^{-1}\zg} S^{-1}\zg \otimes_\zg C_* )\\ 
	& = \tau(\D(G)\otimes_{\zg} C_*)\\
	&= i(\tor_u(C_*;\N(G))). 
	\end{align*}
	Thus we see 
	\begin{equation}\label{eq:torsion} 
	\PP(\tor_u(C_*;\N(G))) = \PP'(\tau(S^{-1}C_*)) = 0,
	\end{equation}
	which completes the proof.
\end{proof}

The following is the main result of this section.

\begin{thm}[Vanishing $L^2$-torsion polytope]\label{thm:vanishing polytope}
	Let $G$ be a group of type $F$ which is of $P\geq 0$-class. Suppose that $G$ contains a non-abelian elementary amenable normal subgroup. Then $G$ is $L^2$-acyclic and we have
	\[ P(G) = 0.\]
\end{thm}
\begin{proof}
	The group $G$ is $L^2$-acyclic by \cite[Theorem 1.44]{Lueck2002}. Let $N$ be the non-abelian elementary amenable normal subgroup.
	
	\emph{Case 1:} $N$ is not virtually abelian. It follows from the proof of \cite[Theorem 2.3.15]{Wegner2000} and the references given therein that $N$ is solvable-by-finite. Hence $N$ has a unique maximal solvable normal subgroup of finite index, say $S$. Since we assume that $N$ is not virtually abelian, $S$ is not abelian. Hence the lowest non-trivial subgroup $A$ in the derived series of $S$ is abelian and contained in $[S,S]\subset [G,G]$. In particular, $A\cap\ker(\pr\colon G\to H_1(G)_f)\neq 0$. Since $A$ is characteristic in $S$ and $S$ is characteristic in $N$, $A$ is normal in $G$.
	
	\emph{Case 2:} $N$ is virtually abelian. Let $A$ be a normal abelian subgroup of finite index. By assumption $N$ is not abelian, so $\ker(\pr\colon N\to H_1(N)_f)$ is non-trivial and hence infinite as $G$ is torsion-free. But any infinite subgroup of $N$ must intersect $A$ non-trivially. Thus in particular, $A\cap\ker(\pr\colon G\to H_1(G)_f)\neq 0$.
	
	
	In both cases we may apply \cref{lem:ore}. This provides us with a subset $S\subset \zg$ satisfying the assumptions of \cref{lem:vanishing} for $X = EG$. Hence $P(G) = 0$.
\end{proof}

\begin{cor}[The $L^2$-torsion polytope of elementary amenable groups vanishes]\label{main theorem:el amenable}
	Let $G$ be an amenable group of type $F$ satisfying the Atiyah Conjecture. If $G$ contains a non-abelian elementary amenable normal subgroup, then
	\[ P(G) = 0.\]
	
	In particular, the $L^2$-torsion polytope of any elementary amenable group of type $F$ vanishes.
	
\end{cor}
\begin{proof}
	By \cref{thm:amenable polytope class} an amenable group $G$ of type $F$ satisfying the Atiyah Conjecture is of polytope class. Hence the first statement follows directly from \cref{thm:vanishing polytope}.
	
	For the second statement, recall from \cref{D:amenable} that an elementary amenable group $G$ of type $F$ satisfies the Atiyah Conjecture. Hence $P(G) = 0$ follows from the previous statement provided that $G$ is non-abelian. If $G$ is abelian, then $G$ must be finitely generated free-abelian, so $P(G) = 0$ follows from $\tor_u(G) = 0$ as seen in \cite[Example 2.7]{FriedlLueck2015b}. 
\end{proof}

We emphasize the following remark that was also used in the proof of \cref{thm:vanishing polytope}.

\begin{remark}
	An elementary amenable group of type $F$ (or more generally, with finite cohomological dimension) is in fact virtually solvable by a result of Hillman-Linnell \cite[Corollary 1]{HillmanLinnell1992}.
\end{remark}

\begin{remark}[Generalization to the universal $L^2$-torsion]
	The proof of \cref{main theorem:el amenable} crucially relies on the existence of a partial order on polytope groups even though the original statement does not involve them. One difficulty in proving the corresponding statement for the universal $L^2$-torsion $\tor_u(G)$ lies in the structural deficit of $\Wh^w(G)$ that it does not carry a meaningful partial order.
\end{remark}

\begin{remark}
	\cref{conj:polytope amenable} and thus \cref{thm:vanishing polytope} are inspired by the following list of vanishing results about $L^2$-invariants and related invariants. An infinite amenable has
	\begin{itemize}
		\setlength\itemsep{0mm}
		\item vanishing $L^2$-Betti numbers, see \cite[Theorem 0.2]{CheegerGromov1986}, or \cite[Theorem 7.2 (1) and (2)]{Lueck2002} for a strengthening of this statement;
		\item vanishing $L^2$-torsion (provided that $G$ is of type $F$), see \cite[Theorem 1.3]{LiThom2014};
		\item vanishing rank gradient with respect to a normal chain with trivial intersection (provided that $G$ is finitely generated), see \cite[Theorem 3]{AbertNikolov2012};
		\item vanishing rank gradient with respect to \emph{any} chain (provided that $G$ is finitely presented), see \cite[Theorem 1]{AbertJZNikolov2011};
		\item fixed price $1$ in the theory of cost of groups, see \cite[Theorem 6]{OrnsteinWeiss1980} combined with \cite[Théorème 3]{Gaboriau2000}.
		\item vanishing simplicial volume (provided that $G$ is the fundamental group of a closed connected orientable manifold), see \cite[Section 3.1, Corollary (C)]{Gromov1983}.
	\end{itemize}
\end{remark}

\medskip

\section{Evidence for non-elementary amenable groups}\label{ch:evidence}

In this short final section, we offer some evidence for the validity of \cref{conj:polytope amenable} for amenable groups that are not elementary amenable. The difference between amenable and elementary amenable is delicate. Finding amenable groups which are not elementary amenable was for a long time part of the Neumann-Day problem. Grigorchuk  constructed the first examples of such groups  \cite{Grigorchuk1984} and later provided finitely presented ones \cite{Grigorchuk1998}. At the time of writing, however, it is still open if there are also examples of type $F$.

The following computation is to a great extent based on known results. Our main tool will be norm maps. Given a finitely generated free-abelian group $H$, we denote by $\Map(\Hom(H,\R),\R)$ the group of continuous maps $\Hom(H,\R)\to\R$ equipped with pointwise addition. A polytope $P\in \PPP(H)$ induces a seminorm on $\Hom(H,\R)$ by
\[ \|\phi\|_P = \max\{ \phi(p)-\phi(q)\mid p,q\in P \}.\]
This seminorm behaves well with respect to Minkowski sums in the sense that
\[  \|\phi\|_{P+Q}  = \|\phi\|_P  + \|\phi\|_Q \]
for all $\phi\in \Hom(H,\R)$, which allows us to make the following definition.

\begin{dfn}[Seminorm homomorphism]\label{def:seminorm map}
	We call
	\[ \norm\colon\P(H)\to \Map(\Hom(H,\R),\R),\;\; P-Q\mapsto \|\cdot \|_P - \|\cdot\|_Q\]
	\emph{seminorm homomorphism}. It passes to the quotient $\P_T(H)$ and the induced map
	\[\norm\colon\P_T(H)\to \Map(\Hom(H,\R),\R)\]
	is denoted by the same symbol.
\end{dfn} 

The cornerstone of our argument will be the following theorem.

\begin{thm}\label{little lemma}
	Let $H$ be a finitely generated free-abelian group. Then we have 
	\begin{align*}
		&\ker\big(\norm\colon\P_T(H)\to \Map(\Hom(H,\R),\R)\big) \\
		= &\ker \big(\id+*\colon \P_T(H)\to \P_T(H)\big) \\
		= &  \im\big(\id-*\colon \P_T(H)\to\P_T(H)\big).
	\end{align*}
\end{thm}
\begin{proof}
	This is the content of \cite[Remark 6.2 and Theorem 6.4]{Funke2016}.
\end{proof}

If $G$ is a group, we will identify $\Hom(H_1(G)_f,\R)$ with $H^1(G;\R)$ in the following.

\begin{prop}[$L^2$-torsion polytope of amenable groups]\label{prop: amenable}
	Let $G\neq\Z$ be an amenable group of type $F$ satisfying the Atiyah Conjecture. Then $P(G)$ lies in the kernel of $\norm\colon \P_T(H_1(G)_f)\to \Map(H^1(G;\R),\R)$ and there is a polytope $P\in\PPP_T(H_1(G)_f)$ such that
	\[ P(G) = P-*P. \]
\end{prop}
\begin{proof}
	Let $\pr\colon G\to H_1(G)_f = H$ be the obvious projection. Suppose that $H\neq 0$ since there is nothing to prove otherwise. Let $\phi\colon H\to \Z$ be an epimorphism, and put $K =\ker(\phi\circ\pr\colon G\to\Z)$. Then we have by \cite[Equation (3.26)]{FriedlLueck2015b} and \cite[Lemma 2.6]{FriedlLueck2015} 
	\begin{align*}
	\norm(P(G))(\phi) = -\ct(i^*EG;\N(K))= -\ct(EK;N(K)),
	\end{align*}
	where $\ct(X;\N(K))$ denotes the $L^2$-Euler characteristic of a $K$-space $X$, see \cite[Section 6.6]{Lueck2002}.
	
	As a subgroup of an amenable group, $K$ is itself amenable. Since $G\neq \Z$, $K$ must be infinite. Since infinite amenable groups are $L^2$-acyclic, we see $\ct(EK;N(K)) = 0$. (Note that for this argument it is irrelevant that $i^*EG = EK$ is not a \emph{finite} $K$-CW-complex.) Thus we have 
	\[\norm(P(G))(\phi)  = 0\]
	for all surjective homomorphisms $\phi\colon H\to\Z$. 
	
	As a difference of seminorms $\norm(P(G))$ is homogeneous and continuous. By the homogeneity we have $\norm(P(G))(\phi) = 0$ for all homomorphisms $\phi\colon H\to\Q$, and by the continuity we have $\norm(P(G))(\phi) = 0$ for homomorphisms $\phi\colon H\to\R$. Hence
	\[P(G) \in \ker\big(\norm\colon  \P_T(H)\to \Map(H^1(G;\R),\R)\big).\]
	
	Now by \cref{little lemma} we have $P(G) \in \im\big(\id-*\colon \P_T(H)\to\P_T(H)\big)$. Hence there exists a class $R-S\in\P_T(H)$ such that 
	\[P(G) = R-S - (*R-*S) = R+*S - *(R+*S).\]
	Taking $P = R+*S$ finishes the proof.
\end{proof}
\medskip

\bibliography{bibliography}

\begin{bibdiv}
\begin{biblist}

\bib{AbertJZNikolov2011}{article}{
      author={Abért, M.},
      author={Jaikin-Zapirain, A.},
      author={Nikolov, N.},
       title={The rank gradient from a combinatorial viewpoint},
        date={2011},
     journal={Groups Geom. Dyn.},
      volume={5},
       pages={213\ndash 230},
}

\bib{AbertNikolov2012}{article}{
      author={Abért, M.},
      author={Nikolov, N.},
       title={{Rank gradient, cost of groups and the rank versus Heegaard genus
  problem}},
        date={2012},
     journal={J. European Math. Soc.},
      volume={14},
      number={5},
       pages={1657\ndash 1677},
}

\bib{Bierietal1987}{article}{
      author={Bieri, R.},
      author={Neumann, W.~D.},
      author={Strebel, R.},
       title={A geometric invariant of discrete groups},
        date={1987},
     journal={Invent. Math.},
      volume={90},
      number={3},
       pages={451\ndash 477},
}

\bib{ChaFriedlFunke2015}{article}{
      author={{Cha}, J.~C.},
      author={{Friedl}, S.},
      author={{Funke}, F.},
       title={{The Grothendieck group of polytopes and norms}},
        date={2017},
     journal={M\"unster J. Math.},
      volume={10},
       pages={75\ndash 81},
}

\bib{CheegerGromov1986}{article}{
      author={Cheeger, J.},
      author={Gromov, M.},
       title={{$L_2$-Cohomology and group cohomology}},
        date={1986},
     journal={Topology},
      volume={25},
       pages={189\ndash 215},
}

\bib{DuboisEtal2015}{article}{
      author={{Dubois}, J.},
      author={{Friedl}, S.},
      author={{L{\"u}ck}, W.},
       title={{The $L^2$-Alexander torsion is symmetric}},
        date={2015},
     journal={Alg. Geom. Top.},
      volume={15},
      number={6},
       pages={3599\ndash 3612},
}

\bib{DuboisEtal2015b}{article}{
      author={{Dubois}, J.},
      author={{Friedl}, S.},
      author={{L{\"u}ck}, W.},
       title={{Three flavors of twisted invariants of knots}},
        date={2015},
     journal={Introduction to Modern Mathematics, Advanced Lectures in
  Mathematics},
      volume={33},
       pages={143\ndash 170},
}

\bib{DuboisEtal2014}{article}{
      author={{Dubois}, J.},
      author={{Friedl}, S.},
      author={{L{\"u}ck}, W.},
       title={{The $L^2$-Alexander torsion of 3-manifolds}},
        date={2016},
     journal={J. Topology},
      volume={9},
      number={3},
       pages={889\ndash 926},
}

\bib{Dieudonne1943}{article}{
      author={Dieudonn{\'e}, J.},
       title={Les d\'eterminants sur un corps non commutatif},
        date={1943},
     journal={Bull. Soc. Math. France},
      volume={71},
       pages={27\ndash 45},
}

\bib{FriedlHarvey2007}{article}{
      author={Friedl, S.},
      author={Harvey, S.},
       title={{Non-commutative Multivariable Reidemeister Torsion and the
  Thurston Norm}},
        date={2007},
     journal={Alg. Geom. Top.},
      volume={7},
       pages={755\ndash 777},
}

\bib{FunkeKielak2016}{article}{
      author={{Funke}, F.},
      author={Kielak, D.},
       title={{Alexander and Thurston norms, and the Bieri-Neumann-Strebel
  invariants for free-by-cyclic groups}},
        date={2018},
     journal={Geom. Topol.},
      volume={22},
       pages={2647\ndash 2696},
}

\bib{FriedlLueck2015}{article}{
      author={Friedl, S.},
      author={L{\"u}ck, W.},
       title={{${L}^2$-Euler characteristics and the Thurston norm}},
        date={2016},
     journal={{arXiv}:1609.07805},
}

\bib{FriedlLueck2015b}{article}{
      author={Friedl, S.},
      author={L{\"u}ck, W.},
       title={{Universal {${L}^2$}-torsion, polytopes and applications to
  3-manifolds}},
        date={2017},
     journal={Proc. London Math. Soc.},
      volume={114},
       pages={1114 \ndash  1151},
}

\bib{FriedlLueckTillmann2016}{article}{
      author={Friedl, S.},
      author={Lück, W.},
      author={Tillmann, S.},
       title={Groups and polytopes},
        date={2016},
     journal={{arXiv}:1609.07805},
}

\bib{Friedl2007}{article}{
      author={Friedl, S.},
       title={{Reidemeister torsion, the Thurston norm and Harvey's
  invariants}},
        date={2007},
     journal={Pac. J. Math.},
      volume={230},
       pages={271\ndash 296},
}

\bib{FriedlTillmann2015}{article}{
      author={Friedl, S.},
      author={Tillmann, S.},
       title={Two-generator one-relator groups and marked polytopes},
        date={2015},
     journal={{arXiv}:1501.03489},
}

\bib{Funke2016}{article}{
      author={{Funke}, F.},
       title={{The integral polytope group}},
        date={2016},
     journal={arXiv:1605.01217},
}

\bib{Gaboriau2000}{article}{
      author={Gaboriau, D.},
       title={Co{\^u}t des relations d'{\'e}quivalence et des groupes},
        date={2000},
     journal={Invent. math.},
      volume={139},
      number={1},
       pages={41\ndash 98},
}

\bib{Grigorchuk1984}{article}{
      author={Grigorchuk, R.~I.},
       title={{Degrees of growth of finitely generated groups, and the theory
  of invariant means}},
        date={1984},
     journal={Izv. Akad. Nauk SSSR Ser. Mat.},
      volume={48},
       pages={939\ndash 985. English version in Mathematics of the USSR\ndash
  Izvestiya \textbf{25} (1985), 259\ndash 300},
}

\bib{Grigorchuk1998}{article}{
      author={Grigorchuk, R.~I.},
       title={{An example of a finitely presented amenable group not belonging
  to the class $EG$}},
        date={1998},
     journal={Math. Sb.},
      volume={189},
       pages={79\ndash 100},
}

\bib{Gromov1983}{article}{
      author={Gromov, M.},
       title={{Volume and bounded cohomology}},
        date={1983},
     journal={Inst. Hautes Etudes Sci. Publ. Math.},
      volume={56},
       pages={5\ndash 99},
}

\bib{HillmanLinnell1992}{article}{
      author={Hillman, J.A.},
      author={Linnell, P.A.},
       title={{Elementary amenable groups of finite Hirsch length are
  locally-finite by virtually solvable}},
        date={1992},
     journal={J. Austral. Math. Soc. (Series A)},
      volume={52},
       pages={237\ndash 241},
}

\bib{KrophollerEtal2006}{article}{
      author={Kropholler, P.~H.},
      author={Linnell, P.~A.},
      author={Moody, J.~A.},
       title={{Applications of a new $K$-theoretic theorem to soluble group
  rings}},
        date={1988},
     journal={Proc. Amer. Math. Soc.},
      volume={104},
      number={3},
       pages={675\ndash 6843},
}

\bib{Linnell2006}{book}{
      author={Linnell, P.~A.},
       title={{Noncommutative localization in group rings}},
      series={Non-commutative localization in algebra and topology},
        date={2006},
}

\bib{LueckReich2005}{book}{
      author={L{\"u}ck, W.},
      author={Reich, H.},
       title={{The Baum-Connes and the Farrell-Jones Conjectures in $K$- and
  $L$-theory}},
      series={{Handbook of $K$-theory, Volume 2}},
   publisher={Springer, Berlin},
        date={2005},
}

\bib{LiThom2014}{article}{
      author={Li, H.},
      author={Thom, A.},
       title={{Entropy, Determinants, and $L^2$-Torsion}},
        date={2014},
     journal={J. Amer. Math. Soc.},
      volume={27},
      number={1},
       pages={239–292},
}

\bib{Lueck2002}{book}{
      author={L{\"u}ck, W.},
       title={{$L^2$}-invariants: theory and applications to geometry and
  {$K$}-theory},
      series={Ergebnisse der Mathematik und ihrer Grenzgebiete. 3. Folge. A
  Series of Modern Surveys in Mathematics [Results in Mathematics and Related
  Areas. 3rd Series. A Series of Modern Surveys in Mathematics]},
   publisher={Springer-Verlag, Berlin},
        date={2002},
      volume={44},
}

\bib{OrnsteinWeiss1980}{article}{
      author={Ornstein, D.~S.},
      author={Weiss, B.},
       title={{Ergodic theory of amenable group actions. I: The Rohlin lemma}},
        date={1980},
     journal={Bull. Amer. Math. Soc.},
      volume={2},
      number={1},
       pages={161\ndash 164},
}

\bib{Rosset1984}{article}{
      author={Rosset, S.},
       title={{A vanishing theorem for Euler characteristics}},
        date={1984},
     journal={Math. Z.},
      volume={185},
       pages={211\ndash 215},
}

\bib{Rosenberg1994}{book}{
      author={Rosenberg, J.},
       title={{Algebraic $K$-Theory and Its Applications}},
      series={Graduate Text in Mathematics},
   publisher={Springer},
     address={New York},
        date={1994},
      volume={147},
}

\bib{Schneider1993}{book}{
      author={Schneider, R.},
       title={{Convex bodies: the Brunn-Minkowski theory}},
   publisher={Cambridge Univ. Press},
     address={Cambridge},
        date={1993},
}

\bib{Silvester1981}{book}{
      author={Silvester, R.~J.},
       title={{Introduction to algebraic $K$-theory}},
   publisher={Chapman \& Hall},
     address={London},
        date={1981},
}

\bib{Tamari1957}{incollection}{
      author={Tamari, D.},
       title={A refined classification of semi-groups leading to generalized
  polynomial rings with a generalized degree concept},
   booktitle={{in Johan C. H. Gerretsen and Johannes de Groot, editors,
  \emph{Proceedings of the International Congress of Mathematicians, Amsterdam,
  1954}, volume 3, pp. 439–440, Groningen, 1957}},
}

\bib{Wegner2000}{article}{
      author={{Wegner}, C.},
       title={{$L^2$-invariants of finite aspherical CW-complexes with
  fundamental group containing a non-trivial elementary amenable normal
  subgroup}},
        date={2000},
     journal={Ph.D. thesis, Münster},
}

\bib{Wegner2009}{article}{
      author={{Wegner}, C.},
       title={{$L^2$-invariants of finite aspherical CW-complexes}},
        date={2009},
     journal={manuscripta math.},
      volume={128},
      number={4},
       pages={469\ndash 481},
}

\bib{Wegner2013}{article}{
      author={{Wegner}, C.},
       title={{The Farrell-Jones Conjecture for virtually solvable groups}},
        date={2015},
     journal={J. Topology},
      volume={8},
      number={4},
       pages={975\ndash 1016},
}

\end{biblist}
\end{bibdiv}

\end{document}